\documentclass[3p,12pt]{elsarticle} 
\usepackage[utf8]{inputenc}
\usepackage[T1]{fontenc}
\usepackage{lmodern}

\usepackage{microtype}

\usepackage{amsmath,amssymb,amsthm}

\usepackage{graphicx}
\usepackage{caption}
\usepackage{subcaption}
\usepackage{float}
\usepackage{booktabs}
\usepackage{multirow}
\usepackage{enumitem}

\usepackage[hidelinks]{hyperref}
\hypersetup{pdftitle={Analytic conjugation between planar reversible and Hamiltonian systems},%
            pdfauthor={F. J. S. Nascimento}}

\usepackage{siunitx}

\theoremstyle{plain}
\newtheorem{theorem}{Theorem}[section]
\newtheorem{lemma}[theorem]{Lemma}

\theoremstyle{definition}

\theoremstyle{remark}

\makeatletter
\def\ps@pprintTitle{%
 \let\@oddhead\@empty
 \let\@evenhead\@empty
 \def\@oddfoot{}%
 \def\@evenfoot{}}
\makeatother

\begin{document}

\begin{frontmatter}

\title{Analytic conjugation between planar reversible and Hamiltonian systems}

\author[1]{F.J. S. Nascimento\corref{cor1}}
\ead{francisco.jsn@univasf.edu.br}
\cortext[cor1]{Corresponding author}
\address[1]{Universidade Federal do Vale do S\~ao Francisco, Colegiado de Geologia, Campus Senhor do Bonfim, BA, Brazil}

\begin{abstract}
In this work we study the local structure of analytic planar vector fields
that are reversible with respect to the linear involution $R(u,v)=(u,-v)$.
We show that every analytic reversible vector field with a nondegenerate equilibrium
is locally analytically conjugate to a Hamiltonian system. More precisely, we prove that,
in a neighbourhood of the origin, the system is analytically equivalent to a Hamiltonian
vector field whose Hamiltonian assumes the classical normal form associated with the type
of the equilibrium: $H(x,y)=F(x^{2}+y^{2})$ in the elliptic case and
$H(x,y)=-F(x^{2}-y^{2})$ in the hyperbolic case, where $F$ is real-analytic and completely
determined by the dynamics. We also show that the conjugacy can be chosen equivariant,
that is, commuting with the reversing involution.

We further discuss the problem of \emph{global} equivalence, which in general remains
open, even in the planar case. In dimensions greater than $2$ the situation becomes
even more delicate: the equivalence between reversible and Hamiltonian systems is
known only at the \emph{formal} level, and the existence of an analytic conjugacy,
even locally, is still a widely open problem.
\end{abstract}


\end{frontmatter}

\section{Introduction and main result}

Let $\Omega$ be an open subset of $\mathbb{R}^2$ and denote by
$C^\omega(\Omega,\mathbb{R}^d)$ the set of real-analytic functions defined on $\Omega$
with values in $\mathbb{R}^d$, for $d\in\{1,2\}$. Let
$P,Q\in C^\omega(\Omega,\mathbb{R})$ and consider the differential system
\begin{equation}\label{eq1}
\dot{x}=P(x,y),\qquad \dot{y}=Q(x,y), \qquad (x,y)\in \Omega.
\end{equation}
The system \eqref{eq1} determines on $\Omega$ the planar vector field
$X=(P,Q)\in C^\omega(\Omega,\mathbb{R}^2)$.

A point \(p\in\Omega\) such that \(X(p)=(0,0)\) is called a \emph{singular point} of \(X\).
We say that \(p\) is \emph{nondegenerate} if the determinant of the Jacobian matrix
\(DX(p)\) is nonzero. In this situation the linear classification yields four possible
types of equilibria:

\begin{itemize}
    \item a \textbf{sink} (or attracting focus), when the eigenvalues of \(DX(p)\)
    have negative real part; in this case there exists a neighbourhood \(V_p\) such that
    every solution starting in \(V_p\) converges to \(p\);

    \item a \textbf{source} (or repelling focus), when the eigenvalues of \(DX(p)\)
    have positive real part; then there exists a neighbourhood \(V_p\) such that every
    solution in \(V_p\setminus\{p\}\) moves away from \(p\);

    \item a \textbf{saddle}, when the eigenvalues of \(DX(p)\) are real and of opposite signs;

    \item a \textbf{center}, when the eigenvalues of \(DX(p)\) are purely imaginary and there
    exists a neighbourhood \(V_p\) such that every solution in \(V_p\setminus\{p\}\) is periodic
    and surrounds \(p\).
\end{itemize}

\medskip

An analytic planar system is called \emph{Hamiltonian} if there exists a function
\(H\in C^\omega(\Omega,\mathbb{R})\) such that
\[
\dot{x}=\partial_y H(x,y),\qquad
\dot{y}=-\partial_x H(x,y).
\]
The function \(H\) is the \emph{Hamiltonian}, and the vector field
\[
X_H(x,y)=\bigl(\partial_y H(x,y),\, -\partial_x H(x,y)\bigr)
\]
is the associated Hamiltonian vector field. The Hamiltonian is always a first integral,
since \(X_H\cdot\nabla H\equiv 0\). Moreover, in the plane the only possible nondegenerate
equilibria are centers (associated with local minima of \(H\)) and saddles
(associated with local maxima of \(H\)); see, for instance,
\cite{jarque1994structural, mello2004saddle}.

\medskip

A linear map $R:\mathbb{R}^2\to \mathbb{R}^2$ is called an \emph{involution} if
$R\neq \mathrm{id}$ and $R^2=\mathrm{id}$. A vector field
$X\in C^\omega(\Omega,\mathbb{R}^2)$ is said to be \emph{reversible with respect to $R$} if
\[
DR(q)\,X(q) = -X\bigl(R(q)\bigr), \qquad \text{for all } q\in\Omega.
\]

Reversibility with respect to a linear involution imposes strong restrictions on the linear
part of the system. In particular, it prevents the presence of eigenvalues with nonzero real
part; hence every nondegenerate equilibrium of a reversible system is necessarily either a
\emph{center} or a \emph{saddle}. For more details, we refer to
\cite{ devaney1976reversible,teixeira1997singularities,teixeira2011reversible}.

An analytic diffeomorphism $h:\Omega\to\Omega$ is said to \emph{commute with $R$} if
\[
h\circ R = R\circ h \quad \text{on } \Omega.
\]

\medskip

Let $\Omega_1,\Omega_2\subset\mathbb{R}^2$ be open sets. Two vector fields
$X_1\in C^\omega(\Omega_1,\mathbb{R}^2)$ and
$X_2\in C^\omega(\Omega_2,\mathbb{R}^2)$ are \emph{analytically conjugate}
if there exists an analytic diffeomorphism $h:\Omega_1\to\Omega_2$ such that
\begin{equation}\label{eq3}
D_q h \, X_1(q) = X_2\bigl(h(q)\bigr), \qquad q\in\Omega_1.
\end{equation}
In this case, $h$ sends singular points to singular points and carries periodic orbits
to periodic orbits, preserving their periods. We say that $X_1$ and $X_2$ are
\emph{locally} analytically conjugate at a singular point $p$ if \eqref{eq3} holds in
a neighbourhood of $p$.

\medskip

These preliminaries allow us to state the main result of this work.

\begin{theorem}\label{thm:equiv_reversivel_hamiltoniano}
Let $X(u,v)$ be an analytic vector field defined in a neighbourhood of $(0,0)$, with
$X(0,0)=(0,0)$ and $d=\det DX(0,0)\neq 0$, and suppose that $X$ is reversible with respect
to the involution
\[
R(u,v)=(u,-v).
\]
Then $X$ is analytically conjugate, in a neighbourhood of $(0,0)$, to a Hamiltonian vector field
\[
X_H(x,y)=\bigl(\partial_y H(x,y),\,-\partial_x H(x,y)\bigr),
\]
whose Hamiltonian can be written, according to the type of the singular point, as
\[
H(x,y)=
\begin{cases}
F(x^{2}+y^{2}), & d<0 \quad (\text{center case}),\\[6pt]
-F(x^{2}-y^{2}), & d>0 \quad (\text{saddle case}),
\end{cases}
\]
where
\[
F(t)=\dfrac{1}{2}\int_{0}^{t} G(s)\,ds,
\qquad G(0)>0,
\]
and $G$ is real-analytic in a neighbourhood of $0$.

Moreover, the conjugacy can be chosen \emph{equivariant} with respect to $R$: there exists
a local analytic diffeomorphism \(h\) such that
\[
Dh(u,v)\,X(u,v)=X_H\bigl(h(u,v)\bigr)
\quad\text{and}\quad
h\circ R = R\circ h.
\]
\end{theorem}

The relationship between reversibility and Hamiltonian structure has been extensively
investigated in the literature. In higher dimensions, the equivalence between reversible and
Hamiltonian systems is known mainly at the \emph{formal} level; see, for example,
Lamb and Roberts \cite{lamb1998time}, Buzzi and Teixeira \cite{buzzi2004time},
as well as the results of Lamb et al.~\cite{lamb2020hamiltonian} and of Martins and
Teixeira \cite{martins2011similarity}. In the plane, Proposition~3.1 of
\cite{martins2011formal} states that every analytic system reversible with respect to
$R(x,y)=(x,-y)$ is \emph{formally} conjugate to a Hamiltonian system.
Further results on the role of reversibility in the center problem can be found in
\cite{gine2011reversibility}.

Theorem~\ref{thm:equiv_reversivel_hamiltoniano} strengthens this picture by replacing
formal equivalence with \emph{analytic} equivalence. We show that, in the plane,
reversibility completely determines the local structure of the system: every analytic
reversible vector field with a nondegenerate equilibrium is analytically conjugate to a
Hamiltonian system, and this conjugacy can be chosen equivariant with the reversing
involution.

\section{Proof of Theorem~\ref{thm:equiv_reversivel_hamiltoniano} }

\begin{lemma}\label{prop:forma_matriz}
Let $X(u,v)$ be an analytic vector field defined in a neighbourhood of $(0,0)$, with
$X(0,0) = (0,0)$ and $d=\det DX(0,0)\neq 0$. Suppose that $X$ is reversible with respect
to the involution \(R(u,v) = (u,-v)\). Then the Jacobian matrix
$A := DX(0,0)$ necessarily has the form
\begin{equation}\label{eq:matriz_A}
A =
\begin{pmatrix}
0 & \beta \\
\gamma & 0
\end{pmatrix},
\qquad \text{with } \beta\gamma\neq 0.
\end{equation}
Moreover, if $\beta\gamma<0$ the origin is a nondegenerate linear center; if
$\beta\gamma>0$ the origin is a nondegenerate saddle point.
\end{lemma}

\begin{proof}
The reversibility condition
\begin{equation}\label{eq:reversibilidade}
DR(u,v)\,X(u,v) = -X\bigl(R(u,v)\bigr)
\end{equation}
holds for all $(u,v)$ close to $(0,0)$. Evaluating at $(0,0)$ and using that $R$ is linear,
we obtain the fundamental matrix relation
\begin{equation}\label{eq:relacao_RA}
R\,A = -A\,R,
\end{equation}
where $R = DR(0,0)$ and $A = DX(0,0)$. In our case,
\[
R = \begin{pmatrix}
1 & 0\\
0 & -1
\end{pmatrix}.
\]

Writing
\[
A = \begin{pmatrix}
\alpha & \beta\\
\gamma & \delta
\end{pmatrix},
\]
relation \eqref{eq:relacao_RA} yields
\[
\begin{pmatrix}
\alpha & \beta\\
-\gamma & -\delta
\end{pmatrix}
=
-
\begin{pmatrix}
\alpha & -\beta\\
\gamma & -\delta
\end{pmatrix}
=
\begin{pmatrix}
-\alpha & \beta\\
-\gamma & \delta
\end{pmatrix}.
\]
Comparing entries we obtain
\[
\alpha = -\alpha \implies \alpha = 0,
\qquad
-\delta = \delta \implies \delta = 0.
\]
Hence
\[
A = \begin{pmatrix}
0 & \beta\\[2pt]
\gamma & 0
\end{pmatrix}.
\]
Since $d=\det A\neq 0$, it follows that $\beta\gamma\neq 0$, which proves
\eqref{eq:matriz_A}.

\medskip

The characteristic polynomial of $A$ is
\[
p(\lambda)=\lambda^{2}-\beta\gamma,
\]
and thus
\[
\lambda_\pm = \pm\sqrt{\beta\gamma}.
\]
Therefore:
\[
\beta\gamma<0 \iff \lambda_\pm = \pm i\omega,\ \omega>0,\quad\text{(linear center)},
\]
\[
\beta\gamma>0 \iff \lambda_\pm\in\mathbb{R}\setminus\{0\},\ \lambda_+\lambda_-<0,\quad\text{(saddle)}.
\]

\medskip

Finally, although purely imaginary eigenvalues do not in general guarantee that the origin
is a center (they may correspond to a focus), the presence of a reversing symmetry rules
out spiralling behaviour. Thus, in the case $\beta\gamma<0$, the origin is indeed a center;
see \cite[Theorem 1]{teixeira2001center}.
\end{proof}

We now split the proof of Theorem~\ref{thm:equiv_reversivel_hamiltoniano} into two cases:
when \(X\) has a center and when \(X\) has a saddle.

\subsection{Center case}

This case follows directly from the Poincaré normal form for nondegenerate analytic centers.
We recall the following classical result; see, for example,
\cite[Theorem~3.3]{chavarriga1999survey}.

\begin{lemma}\label{caso-centro}
Let $X(u,v)$ be an analytic vector field with a nondegenerate center at the origin.
Then $X$ is analytically conjugate, in a neighbourhood of $(0,0)$, to a Hamiltonian
vector field
\[
X_H(x,y)=\bigl(\partial_y H(x,y),\,-\partial_x H(x,y)\bigr),
\]
whose Hamiltonian is of the form
\[
H(x,y)=F(x^{2}+y^{2}),
\]
where
\[
F(t)=\frac{1}{2}\int_{0}^{t}G(s)\,ds,
\]
and $G$ is a real-analytic function defined in a neighbourhood of $0$ satisfying
\(G(0)>0.\)
\end{lemma}

In this case reversibility is not needed to guarantee the existence of an analytic
conjugacy between $X$ and a Hamiltonian vector field. However, if the field $X$ is
$R$-reversible, one can choose, among the infinitely many analytic conjugacies provided
by the Poincaré normal form, an \emph{equivariant} conjugacy with respect to $R$.
More precisely, one can select an analytic diffeomorphism $h$ such that
\[
Dh(u,v)\,X(u,v)=X_H\bigl(h(u,v)\bigr),
\qquad
h\circ R = R\circ h.
\]
Furthermore, such an equivariant conjugacy can be constructed uniquely (that is, with no
freedom in the free coefficients) by a procedure analogous to the one used in the saddle
case presented in the next subsection.

\subsection{Saddle case}

For this case, as far as we know, there is no result in the literature fully analogous to
the Poincaré Normal Form Theorem for centers (Lemma~\ref{caso-centro}), although it is
plausible that such a statement could be derived as a consequence of more general results
in normal form theory.

We therefore present a direct proof of the analytic conjugacy in the case where the origin
is a saddle point.

\begin{lemma}\label{lem:sela}
Let $X(u,v)$ be an analytic vector field defined in a neighbourhood of $(0,0)$ such that
\[
\det DX(0,0)=\beta\gamma>0,
\]
and suppose that $X$ is reversible with respect to 
\[
R(u,v) = (u,-v).
\]
Then $X$ is analytically conjugate, in a neighbourhood of $(0,0)$, to a Hamiltonian
vector field
\[
X_H(x,y)=\bigl(\partial_y H(x,y),\,-\partial_x H(x,y)\bigr),
\]
whose Hamiltonian is of the form
\[
H(x,y)=-F(x^{2}-y^{2}),
\]
where
\[
F(t)=\frac{1}{2}\int_{0}^{t}G(s)\,ds,
\]
and $G$ is a real-analytic function defined in a neighbourhood of $0$ satisfying
\(G(0)>0.\)
\end{lemma}

\begin{proof}
The argument follows the classical scheme of proofs based on analytic normal forms and
is divided into two steps.

In the first step we construct what we shall call a \emph{formal conjugacy}: we show that
it is possible to transform the original system into the desired one by means of power
series, without worrying, at this stage, about the convergence of the series appearing in
the process.

In the second step we use the method of \emph{Cauchy majorants} to prove that the formal
series obtained in the first step actually converge in a neighbourhood of the origin,
thereby completing the construction of the analytic conjugacy; see, for example,
\cite{bibikov2006local, chow1994normal}.

\subsubsection{Formal conjugacy}

Since $\det DX(0,0)=\beta\gamma>0$, after a linear change of coordinates we may assume
$\beta=\gamma=1$. Thus the differential system associated with $X$ can be written as
\begin{equation}\label{caso-sela}
\begin{cases}
\dot{u}=v+P(u,v),\\[2pt]
\dot{v}=u+Q(u,v),
\end{cases}
\end{equation}
where $P$ and $Q$ are analytic functions in a neighbourhood of $(0,0)$ satisfying
\[
P(0,0)=Q(0,0)=\partial_uP(0,0)=\partial_vP(0,0)
=\partial_uQ(0,0)=\partial_vQ(0,0)=0,
\]
and
\[
P(u,-v)=-P(u,v),\qquad Q(u,-v)=Q(u,v).
\]

The linear matrix associated with \eqref{caso-sela} is
\[
A=\begin{pmatrix}
0 & 1\\
1 & 0
\end{pmatrix},
\]
whose eigenvalues are $-1$ and $1$. In particular, $A$ is diagonalizable: there exist
matrices $B$ (diagonal) and $C$ (invertible) such that $B=C^{-1}AC$. A convenient choice is
\[
B=
\begin{pmatrix}
-1 & 0\\
0 & 1
\end{pmatrix},
\qquad
C=
\begin{pmatrix}
1 & 1\\
-1 & 1
\end{pmatrix},
\qquad
C^{-1}=
\begin{pmatrix}
\frac{1}{2} & -\frac{1}{2}\\[2pt]
\frac{1}{2} & \frac{1}{2}
\end{pmatrix}.
\]
The linear change of variables
\[
(u,v)=C(\xi,\eta)
\]
transforms system~\eqref{caso-sela} into the new system
\begin{equation}\label{sist_2}
\begin{cases}
\dot{\xi}=-\xi+f_1(\xi,\eta),\\[2pt]
\dot{\eta}=\eta+f_2(\xi,\eta),
\end{cases}
\end{equation}
where $f_1$ and $f_2$ are analytic functions in a neighbourhood of $(0,0)$ with
\[
f_i(0,0)=\partial_{\xi}f_i(0,0)=\partial_{\eta}f_i(0,0)=0,\qquad i=1,2.
\]
More precisely,
\[
f_1(\xi,\eta)=P(\xi+\eta,-\xi+\eta)-Q(\xi+\eta,-\xi+\eta),
\]
\[
f_2(\xi,\eta)=P(\xi+\eta,-\xi+\eta)+Q(\xi+\eta,-\xi+\eta).
\]

Since system~\eqref{caso-sela} is reversible with respect to the linear involution
\[
R=
\begin{pmatrix}
1 & 0\\
0 & -1
\end{pmatrix},
\]
system~\eqref{sist_2} is reversible with respect to the linear involution
\[
S=C^{-1}RC=
\begin{pmatrix}
0 & 1\\
1 & 0
\end{pmatrix},
\]
that is, $S(\xi,\eta)=(\eta,\xi)$. In particular, reversibility yields the relation
\begin{equation}\label{condS}
f_2(\xi,\eta)=-f_1\circ S(\xi,\eta)=-f_1(\eta,\xi).
\end{equation}

Note that the linear change $(u,v)=C(\xi,\eta)$ transforms a system reversible with respect
to $R(u,v)=(u,-v)$ into a system reversible with respect to $S(\xi,\eta)=(\eta,\xi)$; that
is, the reversing involution is modified. In general, an arbitrary change of coordinates
does not preserve the original involution. Our next goal is to construct a change of
variables
\[
(\xi,\eta)\longmapsto h(\xi,\eta)
\]
that \emph{preserves} the involution $S$. In this way, the new system obtained will again
be $S$-reversible. For this we use the following result.

\begin{lemma}\label{auxlem}
Let $X$ be the vector field associated with system~\eqref{sist_2} and
\[
h(\xi,\eta)=\bigl(\xi+h_1(\xi,\eta),\ \eta+h_2(\xi,\eta)\bigr)
\]
a diffeomorphism such that $h\circ S=S\circ h$. Then the change of variables
$(\xi,\eta)\mapsto h(\xi,\eta)$ transforms the $S$-reversible field $X$ into a new
$S$-reversible field $\widehat{X}$.
\end{lemma}

In other words, changes of variables that commute with $S$ preserve the reversibility of
the original field.

\begin{proof}
A proof of this lemma can be found in \cite[Lemma~IV]{gaeta1994normal}. We remark that in
that paper G.~Gaeta proves the existence of normal forms for reversible systems in a much more general context than the one considered here.
\end{proof}

\medskip
Our construction will be carried out by induction. Let
\[
X_1(\xi,\eta)=B(\xi,\eta)+X^{2}_1(\xi,\eta)+\cdots
\]
be the power–series expansion of the vector field $X_1$ associated with
system~\eqref{sist_2}. Since $X_1$ is reversible with respect to
$S(\xi,\eta)=(\eta,\xi)$, the coordinate functions of $X^{2}_1$ satisfy
relation~\eqref{condS}. Hence $X^{2}_1$ can be written in the form
\[
X^2_1(\xi,\eta)=\bigl(
 -a_{20}\xi^2-a_{11}\xi\eta-a_{02}\eta^2,\,
 a_{02}\xi^2+a_{11}\xi\eta+a_{20}\eta^2
\bigr),
\]
where the coefficients $a_{kj}$ are determined by the coefficients of the
Taylor series of $f_1$.

Similarly, we write the power–series expansion of $h$ as
\[
h(\xi,\eta)=(\xi,\eta)+h^{2}(\xi,\eta)+\cdots.
\]
Since we want $h$ to commute with $S$, we must impose
\[
h\circ S=S\circ h,
\]
which is equivalent to
\begin{equation}\label{condh}
h_2(\xi,\eta)=h_1\circ S(\xi,\eta)=h_1(\eta,\xi).
\end{equation}
Consequently, $h^{2}$ can be written as
\[
h^2(\xi,\eta)=\bigl(
b_{20}\xi^2+b_{11}\xi\eta+b_{02}\eta^2,\,
b_{02}\xi^2+b_{11}\xi\eta+b_{20}\eta^2
\bigr),
\]
where the coefficients $b_{kj}$ are, a priori, undetermined.

The homological operator associated with the linear part $B$,
denoted by $L_B^{(2)}h^2$, is defined (see Section~2.3 of Chapter~1 in
\cite{nascimento2023sistemas}) by
\[
(L_B^{(2)}h^2)(\xi,\eta)
= D h^2(\xi,\eta)\, B(\xi,\eta) - B\, h^2(\xi,\eta).
\]
Substituting the explicit form of $h^2$ and carrying out the corresponding
computations, we obtain
\[
(L_B^{(2)}h^2)(\xi,\eta)
=
\bigl(
 -b_{20}\,\xi^{2}
 + b_{11}\,\xi\eta
 + 3 b_{02}\,\eta^{2},
\;
 -3 b_{02}\,\xi^{2}
 - b_{11}\,\xi\eta
 + b_{20}\,\eta^{2}
\bigr).
\]

Imposing
\[
(L_B^{(2)}h^2)(\xi,\eta)=X^2_1(\xi,\eta),
\]
we obtain
\[
b_{20}=a_{20},\qquad b_{11}=-a_{11},\qquad b_{02}=-\frac{a_{02}}{3}.
\]
Therefore all coefficients of $h^2$ are determined by the coefficients of
$X^2_1$. This implies that the change of variables
\[
(\xi,\eta)\mapsto(\xi,\eta)+h^2(\xi,\eta)
\]
transforms the field $X_1$ into a new field $X_2$ of the form
\[
X_2(\xi,\eta)=B(\xi,\eta)+X^3_2(\xi,\eta)+\cdots.
\]
Moreover, since $h^2(\xi,\eta)$ commutes with $S$, Lemma~\ref{auxlem} guarantees
that $X_2$ is reversible with respect to~$S$. Thus $X^3_2$ can be written as
\[
X^3_2(\xi,\eta)=\bigl(
 -a_{30}\xi^3-a_{21}\xi^2\eta-a_{12}\xi\eta^2-a_{03}\eta^3,\,
 a_{03}\xi^3+a_{12}\xi^2\eta+a_{21}\xi\eta^2+a_{30}\eta^3
\bigr).
\]

Repeating the same procedure, we take
\[
h^3(\xi,\eta)=\bigl(
b_{30}\xi^3+b_{21}\xi^2\eta+b_{12}\xi\eta^2+b_{03}\eta^3,\,
b_{03}\xi^3+b_{12}\xi^2\eta+b_{21}\xi\eta^2+b_{30}\eta^3
\bigr)
\]
and obtain
\[
(L_B^{(3)}h^3)(\xi,\eta)=\bigl(
-2 b_{30}\xi^3+2b_{12}\xi\eta^2+4b_{03}\eta^3,\,
-4 b_{03}\xi^3-2b_{12}\xi^2\eta+2b_{30}\eta^3
\bigr).
\]
Imposing now
\[
(L_B^{(3)}h^3)(\xi,\eta)=X^3_2(\xi,\eta),
\]
we arrive at
\[
b_{30}=\frac{a_{30}}{2},\qquad b_{12}=-\frac{a_{12}}{2},\qquad b_{03}=-\frac{a_{03}}{4}.
\]

Observe that in this case we cannot determine the coefficient $b_{21}$ in terms of the
coefficients of $X^3_2$. This means that the coefficient $a_{21}$ cannot be removed by
the change of variables. Thus the change
\[
(\xi,\eta)\mapsto(\xi,\eta)+h^3(\xi,\eta)
\]
transforms the field $X_2$ into a new field $X_3$ of the form
\[
X_3(\xi,\eta)=B(\xi,\eta)+\bigl(-a_{21}\xi^2\eta,\ a_{21}\xi\eta^2\bigr)+X^4_3(\xi,\eta)+\cdots.
\]
Moreover, since $h^3(\xi,\eta)$ commutes with $S$, Lemma~\ref{auxlem} again ensures that
$X_3$ is reversible with respect to~$S$.

\medskip

The above procedure can be iterated by induction for every $k\geq 2$. At step $k\geq 2$
the polynomial $h^{k}$ is determined by the involution $S$ and by the coefficients of the
polynomial $X^k_{k-1}$, which have already been fixed at the previous step. Applying the
change of variables
\[
(\xi,\eta)\mapsto(\xi,\eta)+h^{k}(\xi,\eta)
\]
to the field $X_{k-1}$ we obtain a new field $X_k$, which is, by construction, reversible
with respect to $S$. Since the procedure works for $k=2$, by the principle of mathematical
induction it is valid for all $k\in\mathbb{N}$ with $k\geq 2$.

\medskip

We now derive the general expression of $h^{k}$ from $X^k_{k-1}$, for $k\geq 2$. Since
$X^k_{k-1}$ is $S$-reversible and $h^{k}$ commutes with $S$, we can write
\[
X^k_{k-1}(\xi,\eta)=\Bigl(
 -\sum_{j=0}^k a_{k-j,j}\,\xi^{k-j}\eta^j,\ 
 \sum_{j=0}^k a_{j,k-j}\,\xi^{k-j}\eta^j
\Bigr)
\]
and
\[
h^k(\xi,\eta)=\Bigl(
 \sum_{j=0}^k b_{k-j,j}\,\xi^{k-j}\eta^j,\ 
 \sum_{j=0}^k b_{j,k-j}\,\xi^{k-j}\eta^j
\Bigr),
\]
where the coefficients $a_{k-j,j}$ are determined by the previous steps and the
$b_{j,k-j}$ are, a priori, undetermined.

The homological operator
\[
(L_B^{(k)}h^k)(\xi,\eta)=Dh^k(\xi,\eta)\,B(\xi,\eta)-B\,h^k(\xi,\eta)
\]
takes, after a direct computation, the form
\[
(L_B^{(k)}h^k)(\xi,\eta)=\Biggl(
 \sum_{j=0}^k (-k+2j+1)\,b_{k-j,j}\,\xi^{k-j}\eta^j,\ 
 \sum_{j=0}^k (-k+2j-1)\,b_{j,k-j}\,\xi^{k-j}\eta^j
\Biggr).
\]
Imposing
\[
(L_B^{(k)}h^k)(\xi,\eta)=X^k_{k-1}(\xi,\eta),
\]
we obtain, for every $k\geq 2$ and every $0\leq j\leq k$,
\[
\begin{cases}
(-k+2j+1)\,b_{k-j,j}=-a_{k-j,j},\\[2pt]
(-k+2j-1)\,b_{j,k-j}=a_{j,k-j}.
\end{cases}
\]

Note that these two equations yield the same coefficients, so it suffices to solve one of
them. From the first equation we see that $-k+2j+1=0$ if and only if $k=2j+1$, that is,
if $k$ is odd. Thus, if $k$ is even we obtain
\[
b_{k-j,j}=\frac{a_{k-j,j}}{k-2j-1},
\]
for every $k\geq 2$ and $0\leq j\leq k$. This shows that when $k$ is even the polynomials
$X^k_{k-1}$ are completely removed by the change of variables
$(\xi,\eta)\mapsto(\xi,\eta)+h^{k}(\xi,\eta)$ (as exemplified in the case $k=2$).

Now suppose that $k$ is odd, that is, $k=2m+1$ with $m\geq 1$. Then
\[
-k+2j+1=-2(m-j),
\]
and consequently
\[
b{\,}_{2m+1-j,j}=\frac{a_{2m+1-j,j}}{2(m-j)},
\]
for every $j\neq m$. The terms of $X^k_{k-1}$ associated with these coefficients are also
removed by the change of coordinates. However, for $j=m$ the coefficient $b_{m+1,m}$ cannot
be expressed in terms of $a_{m+1,m}$; hence the terms of $X^k_{k-1}$ of the form
\[
\bigl(-a_{m+1,m}\,\xi^{m+1}\eta^m,\ a_{m+1,m}\,\xi^m\eta^{m+1}\bigr)
\]
cannot be eliminated by the change of variables
$(\xi,\eta)\mapsto(\xi,\eta)+h^{k}(\xi,\eta)$ when $k$ is odd. This is exactly what
happens in the case $k=3$. When $k=2m+1$, the coefficients $b_{m+1,m}$ remain free, that
is, we may assign them arbitrary values; for simplicity we shall assume $b_{m+1,m}=0$ for
all $m\geq 1$.

\medskip

Let
\[
h(\xi,\eta)=\bigl(\xi+h_1(\xi,\eta),\ \eta+h_2(\xi,\eta)\bigr)
\]
be the formal diffeomorphism constructed by the above procedure. Since $h$ commutes with
$S$, the change of variables
\begin{equation}\label{coord}
(\xi,\eta)\longmapsto\bigl(\xi+h_1(\xi,\eta),\ \eta+h_2(\xi,\eta)\bigr)
\end{equation}
transforms the $S$-reversible system~\eqref{sist_2} into a new $S$-reversible system that
can be written as
\begin{equation}\label{sist_4}
\begin{cases}
\dot{\xi}=-\xi-\displaystyle\sum_{m\geq 1} g_m(\xi\eta)^m\,\xi,\\[6pt]
\dot{\eta}=\eta+\displaystyle\sum_{m\geq 1} g_m(\xi\eta)^m\,\eta,
\end{cases}
\end{equation}
where the $g_m$ are precisely the coefficients $a_{m+1,m}$ that were not eliminated by the
changes of variables. Defining $g:\mathbb{R}\to\mathbb{R}$ by
\begin{equation}\label{dfg}
g(t)=\sum_{m\geq 1} g_m t^m,
\end{equation}
we can rewrite system~\eqref{sist_4} as
\begin{equation}\label{sist_5}
\begin{cases}
\dot{\xi}=-\xi-\xi\,g(\xi\eta),\\[4pt]
\dot{\eta}=\eta+\eta\,g(\xi\eta),
\end{cases}
\end{equation}
with $g(0)=0$.

Finally, the linear change of variables
\[
(x,y)=C^{-1}(\xi,\eta)
\]
transforms system~\eqref{sist_5} into the system
\begin{equation}\label{sist_4_1}
\widehat{X}(x,y):=
\begin{cases}
\dot{x}=y\,G(x^2-y^2),\\[4pt]
\dot{y}=x\,G(x^2-y^2),
\end{cases}
\end{equation}
where
\[
G(t)=1+g\!\left(\frac{t}{4}\right).
\]

Observe that system \eqref{sist_4_1} is Hamiltonian, with Hamiltonian given by
\[
H(x,y)=-F(x^2-y^2),
\qquad
F(t)=\frac{1}{2}\int_0^t G(s)\,ds.
\]

Moreover, system~\eqref{sist_4_1} is reversible with respect to $R(x,y)=(x,-y)$, since,
as at the beginning of the formal conjugacy, we have $S=C^{-1}RC$, and therefore
\[
R = C S C^{-1} =
\begin{pmatrix}
1 & 0\\
0 & -1
\end{pmatrix}.
\]

This completes the construction of the formal conjugacy.

\

In summary, the $R$-reversible field $X(u,v)$ defined by system~\eqref{caso-sela} is
formally conjugate to the $R$-reversible field $\widehat{X}(x,y)$ defined by
\eqref{sist_4_1}. Moreover, the formal diffeomorphism $(x,y)=\overline{h}(u,v)$ that
conjugates $X$ and $\widehat{X}$ commutes with $R$. In fact, by construction,
\[
\overline{h}(u,v)=C\circ h\circ C^{-1}(u,v),
\]
where $h(\xi,\eta)$ is the diffeomorphism that conjugates systems~\eqref{sist_2} and
\eqref{sist_5}. Since $h\circ S=S\circ h$ and $CS=RC$, we have
\begin{align*}
\overline{h}\circ R(u,v)
&=C\circ h\circ C^{-1}R(u,v)\\
&=C\circ h\circ S\circ C^{-1}(u,v)\\
&=C\circ S\circ h\circ C^{-1}(u,v)\\
&=R\circ C\circ h\circ C^{-1}(u,v)
=R\circ \overline{h}(u,v),
\end{align*}
that is, $(x,y)=\overline{h}(u,v)$ commutes with $R$.

\subsubsection{Convergence}

Observe that, since $C$ and $C^{-1}$ are linear transformations, the convergence of
$\overline{h}(u,v)$ is equivalent to the convergence of $h(\xi,\eta)$; that is,
$\overline{h}(u,v)$ converges if and only if $h(\xi,\eta)$ converges, where
$h(\xi,\eta)$ is the formal diffeomorphism that conjugates
systems~\eqref{sist_2} and~\eqref{sist_5}.

Let
\begin{equation}\label{ss5}
h(\xi,\eta)=\bigl(\xi+h_1(\xi,\eta),\,\eta+h_2(\xi,\eta)\bigr),\qquad
h_1(0,0)=h_2(0,0)=0,
\end{equation}
be the formal diffeomorphism that conjugates systems \eqref{sist_2} and \eqref{sist_5},
and let $g$ be the formal function defined in \eqref{dfg}. Our goal is to show that $h$
and $g$ are convergent in a neighbourhood of the origin.

Since $h$ commutes with $S$, it follows from \eqref{condh} that
\[
h_2(\xi,\eta)=h_1(\eta,\xi),
\]
so the convergence of $h_1$ implies the convergence of $h_2$. Thus the convergence of $h$
reduces to the convergence of $h_1$. We write
\[
f_1(\xi,\eta)=\sum_{k=2}^{\infty} f^k_1(\xi,\eta)
\quad\text{and}\quad
h_1(\xi,\eta)=\sum_{k=2}^{\infty} h^k_1(\xi,\eta)
\]
as the Taylor expansions of $f_1$ and $h_1$, where
\[
f^k_1(\xi,\eta)=\sum_{j=0}^{k} a_{k-j,j}\,\xi^{k-j}\eta^j,
\qquad
h^k_1(\xi,\eta)=\sum_{j=0}^{k} b_{k-j,j}\,\xi^{k-j}\eta^j.
\]

Applying the change of variables and rearranging the terms, we obtain the relation
(cf.~\cite[Chap.~2, \S17]{siegel2012lectures})
\begin{equation}\label{s6}
(-k+2j+1)\,b_{k-j,j}=\tilde{a}_{k-j,j}
-\sum_{m=1}^{n}(k-2j)\,g_m\,b_{k-j-m,j-m},
\end{equation}
where $\tilde{a}_{k-j,j}$ denotes the coefficient of $\xi^{k-j}\eta^{j}$ in the
reexpansion of $f_1\bigl(h(\xi,\eta)\bigr)$ as a power series in $(\xi,\eta)$, and
\[
n=\min\{k-j,j\}.
\]
Thus, for those values of $k$ such that $-k+2j+1\neq 0$, we have
\begin{equation}\label{s7}
b_{k-j,j}=\frac{1}{-k+2j+1}\left(
\tilde{a}_{k-j,j}-\sum_{m=1}^{n}(k-2j)\,g_m\,b_{k-j-m,j-m}
\right).
\end{equation}

Recall that when $k=2m+1$ we have $-k+2j+1=0$ if and only if $j=m$. In this case
$k-2j=2m+1-2j=2(m-j)+1$, and in particular, for $j=m$ we obtain $k-2m=1$. Substituting
$k=2m+1$ into \eqref{s6}, we get
\[
-2(m-j)\,b_{2m+1-j,j}
=\tilde{a}_{2m+1-j,j}-\sum_{m=1}^{n}(2(m-j)+1)\,g_m\,b_{2m+1-j-m,j-m}.
\]
Taking $j=m$, we obtain
\[
0=\tilde{a}_{m+1,m}-g_m b_{1,0}.
\]
Since $b_{1,0}=1$, it follows that
\[
g_m=\tilde{a}_{m+1,m}.
\]
In this case we choose
\[
b_{m+1,m}=0.
\]

Summarizing, we have
\begin{equation}\label{s12}
\begin{cases}
g_m=\tilde{a}_{m+1,m},\quad b_{m+1,m}=0, & \text{if } k=2m+1 \text{ and } j=m,\\[4pt]
\displaystyle
b_{k-j,j}=\frac{\tilde{a}_{k-j,j}}{-k+2j+1}
-\sum_{m=1}^{n}\frac{k-2j}{-k+2j+1}\,g_m\,b_{k-j-m,j-m},
& \text{otherwise.}
\end{cases}
\end{equation}
Therefore the above construction produces a unique pair of formal objects \(g\) and \(h\)
compatible with reversibility and with the structure of the homological equation.

\

The next step is to obtain bounds for $g_m$ and $b_{k-j,j}$. From the first equation in
\eqref{s12} we obtain
\[
|g_m|\leq|\tilde{a}_{m+1,m}|,
\]
and therefore
\begin{equation}\label{mg}
\widehat{g}(\xi\eta):=\sum_{m=1}^\infty |g_m|(\xi\eta)^m
\leq\sum_{m=1}^\infty |\tilde{a}_{m+1,m}|(\xi\eta)^m
\leq\sum_{k=2}^\infty\sum_{j=0}^k |\tilde{a}_{k-j,j}|\xi^{k-j}\eta^j
=:\widehat{f}(\xi,\eta).
\end{equation}

Note that $\widehat{f}$ is an analytic function in a neighbourhood of the origin, since it
is obtained from the function
\[
f(\xi,\eta)=\sum_{k=2}^\infty\sum_{j=0}^k \tilde{a}_{k-j,j}\,\xi^{k-j}\eta^j
\]
by replacing each coefficient $\tilde{a}_{k-j,j}$ with its absolute value. As $f$ is
analytic near the origin (being the composition of $f_1$ with polynomials), it follows
that $\widehat{f}$ is also analytic near the origin. Since $\widehat{g}$ is dominated by
$\widehat{f}$ (\(\widehat{g}\prec\widehat{f}\)), we conclude that $\widehat{g}$ is
analytic in a neighbourhood of $0$ and consequently
\[
g(t)=\sum_{m\geq 1}g_m t^m
\]
is analytic in a neighbourhood of $0$.

\

From the second equation of \eqref{s12} we obtain
\begin{equation}\label{eq_28}
|b_{k-j,j}|
\leq\left|\frac{1}{-k+2j+1}\right|\,|\tilde{a}_{k-j,j}|
+\sum_{m=1}^{n}\left|\frac{k-2j}{-k+2j+1}\right|\,|g_m|\,|b_{k-j-m,j-m}|.
\end{equation}
There exists a constant $c>0$, independent of $k$ and $j$, such that
\[
\left|\frac{1}{-k+2j+1}\right|<c
\quad\text{and}\quad
\left|\frac{k-2j}{-k+2j+1}\right|<c
\]
for all relevant indices. Hence
\[
|b_{k-j,j}|
\leq c\,|\tilde{a}_{k-j,j}|
+c\sum_{m=1}^{n}|g_m|\,|b_{k-j-m,j-m}|.
\]

It then follows that
\begin{align*}
\widehat{h}(\xi,\eta)
&:=\sum_{k=2}^\infty\sum_{j=0}^k |b_{k-j,j}|\,\xi^{k-j}\eta^j\\
&\leq c\sum_{k=2}^\infty\sum_{j=0}^k |\tilde{a}_{k-j,j}|\,\xi^{k-j}\eta^j
+c\sum_{k=2}^\infty\sum_{j=0}^k\sum_{m=1}^{n}|g_m|\,|b_{k-j-m,j-m}|\,\xi^{k-j}\eta^j\\
&\leq c\,\widehat{f}(\xi,\eta)
+c\,\widehat{g}(\xi\eta)\,\widehat{h}(\xi,\eta)\\
&\leq c\,\widehat{f}(\xi,\eta)
+c\,\widehat{f}(\xi,\eta)\,\widehat{h}(\xi,\eta),
\end{align*}
where, in the last inequality, we used \eqref{mg}. Thus
\[
\widehat{h}(\xi,\eta)
\prec
\frac{c\,\widehat{f}(\xi,\eta)}{1-c\,\widehat{f}(\xi,\eta)}
=:F(\xi,\eta),
\]
whenever $c\,\widehat{f}(\xi,\eta)<1$. The function $F$ is analytic in a neighbourhood of
the origin, since it is the composition of the analytic function $x\mapsto x/(1-x)$ (for
$|x|<1$) with $c\,\widehat{f}(\xi,\eta)$. Therefore $\widehat{h}$ is analytic in a
neighbourhood of the origin.

Since $h_1$ is dominated by $\widehat{h}$, it follows that $h_1$ is also analytic in a
neighbourhood of the origin and consequently $h_2$ is analytic (because
$h_2(\xi,\eta)=h_1(\eta,\xi)$).

We thus conclude that the functions $h_i$ ($i=1,2$) and $g$ are convergent in a
neighbourhood of the origin. This completes the proof of Lemma~\ref{lem:sela}.

\end{proof}

\subsubsection{Conclusion of the proof of Theorem~\ref{thm:equiv_reversivel_hamiltoniano}}

The proof naturally splits into two cases, according to the sign of 
\(\det DX(0,0)\), as stated in Lemma~\ref{prop:forma_matriz}.  

\medskip

\noindent\textbf{(i) Center case.}  
In this case the linear part of the field has purely imaginary eigenvalues and, as seen in
Lemma~\ref{prop:forma_matriz}, reversibility rules out the possibility of a focus. Thus
the origin is a nondegenerate center. 

Applying Lemma~\ref{caso-centro}, we obtain an analytic change of coordinates that
transforms the original field into a Hamiltonian field
\[
X_H(x,y)=\bigl(\partial_y H(x,y),\,-\partial_x H(x,y)\bigr),
\qquad
H(x,y)=F(x^2+y^2),
\]
where \(F\) is real-analytic and given by
\[
F(t)=\frac12\int_0^t G(s)\, ds,
\qquad
G(0)>0.
\]
Therefore the theorem is proved in this case.

\medskip

\noindent\textbf{(ii) Saddle case.}  
In this case the linear part of the field has real eigenvalues of opposite signs, so that
the origin is a nondegenerate saddle. By Lemma~\ref{lem:sela} there exists an analytic
change of coordinates that transforms the original field into a Hamiltonian field of the
form
\[
X_H(x,y)=\bigl(\partial_y H(x,y),\,-\partial_x H(x,y)\bigr),
\qquad
H(x,y)=-F(x^2-y^2),
\]
with \(F\) real-analytic and again given by
\[
F(t)=\frac12\int_0^t G(s)\, ds,
\qquad
G(0)>0.
\]

\medskip

In both cases the analytic conjugacy constructed can be chosen so as to preserve the
reversing involution \(R(u,v)=(u,-v)\).

In the saddle case, the very construction of the reversible normal form yields an analytic
diffeomorphism \(h\) satisfying
\[
Dh(u,v)\,X(u,v)=X_H\bigl(h(u,v)\bigr),
\qquad
h\circ R = R\circ h.
\]

In the center case, the construction can be adapted, in a completely analogous way to the
saddle case, to obtain an analytic diffeomorphism \(h\) satisfying the same identities.

\medskip

This completes the proof of Theorem~\ref{thm:equiv_reversivel_hamiltoniano}.

\section{Conclusion}

We have shown that every planar analytic vector field reversible with respect to the
involution $R(u,v)=(u,-v)$ is locally analytically conjugate to a Hamiltonian vector
field whose Hamiltonian function assumes the classical normal form determined by the
type of the equilibrium point. This result substantially strengthens the existing
literature: whereas the equivalence between reversible and Hamiltonian systems in the
plane was previously known only at a \emph{formal} level, as established in
\cite{martins2011formal}, we prove here that such an equivalence can be achieved in the
\emph{analytic} sense and in an \emph{equivariant} way with respect to the reversing
involution.

In the elliptic case, the equivalence can be further refined. As shown in
\cite{dosSantosNascimento2024}, the system is analytically conjugate to a Hamiltonian of
potential type, providing an even more rigid description of the local dynamics.

The question of \emph{global} equivalence remains open in general, even in dimension~2.
Nevertheless, there exist specific classes for which global equivalence can be obtained,
such as the reversible Newtonian systems studied in \cite{nascimento2023sistemas} and
the global normalizations for centers developed in \cite{grotta2025global}.

Thus, the present work completes the local description of the dynamical structure of
planar reversible systems, clarifying their deep connection with Hamiltonian systems and
suggesting that analogous phenomena may occur in higher dimensions — where the theory
remains, for the moment, essentially formal.

\bibliography{Refs} 

@article{teixeira2001center,
  title={The center-focus problem and reversibility},
  author={Teixeira, Marco Antonio and Yang, Jiazhong},
  journal={Journal of Differential Equations},
  volume={174},
  number={1},
  pages={237--251},
  year={2001},
  publisher={Elsevier}
}

@article{teixeira1997singularities,
  title={Singularities of reversible vector fields},
  author={Teixeira, Marco Antonio},
  journal={Physica D: Nonlinear Phenomena},
  volume={100},
  number={1-2},
  pages={101--118},
  year={1997},
  publisher={Elsevier}
}

@article{teixeira2011reversible,
  title={Reversible-equivariant systems and matricial equations},
  author={Teixeira, Marco A and Martins, Ricardo M},
  journal={Anais da Academia Brasileira de Ciências},
  volume={83},
  pages={375--390},
  year={2011}
}

@article{buzzi2004time,
  title={Time-reversible Hamiltonian vector fields with symplectic symmetries},
  author={Buzzi, Claudio A and Teixeira, Marco A},
  journal={Journal of Dynamics and Differential Equations},
  volume={16},
  number={2},
  pages={559--574},
  year={2004},
  publisher={Springer}
}

@article{lamb1998time,
  title={Time-reversal symmetry in dynamical systems: a survey},
  author={Lamb, Jeroen S. W. and Roberts, John A. G.},
  journal={Physica D: Nonlinear Phenomena},
  volume={112},
  number={1-2},
  pages={1--39},
  year={1998},
  publisher={Elsevier}
}

@article{gaeta1994normal,
  title={Normal forms of reversible dynamical systems},
  author={Gaeta, Giuseppe},
  journal={International Journal of Theoretical Physics},
  volume={33},
  pages={1917--1928},
  year={1994},
  publisher={Springer}
}

@article{martins2011similarity,
  title={On the similarity of Hamiltonian and reversible vector fields in 4D},
  author={Martins, Ricardo Miranda and Teixeira, Marco Antonio},
  journal={Communications on Pure and Applied Analysis},
  volume={10},
  number={4},
  pages={1257--1266},
  year={2011}
}

@article{martins2011formal,
  title={Formal equivalence between normal forms of reversible and Hamiltonian dynamical systems},
  author={Martins, Ricardo Miranda},
  journal={Communications on Pure \& Applied Analysis},
  volume={13},
  number={2},
  year={2014}
}

@article{lamb2020hamiltonian,
  title={On the Hamiltonian structure of normal forms at elliptic equilibria of reversible vector fields in $\mathbb{R}^4$},
  author={Lamb, Jeroen S. W. and Lima, Mauricio F. S. and Martins, Ricardo M. and Teixeira, Marco A. and Yang, Jiazhong},
  journal={Journal of Differential Equations},
  volume={269},
  number={12},
  pages={11366--11395},
  year={2020},
  publisher={Elsevier}
}

@article{gine2011reversibility,
  title={The reversibility and the center problem},
  author={Gin{\'e}, Jaume and Maza, Susanna},
  journal={Nonlinear Analysis: Theory, Methods \& Applications},
  volume={74},
  number={2},
  pages={695--704},
  year={2011},
  publisher={Elsevier}
}

@article{devaney1976reversible,
  title={Reversible diffeomorphisms and flows},
  author={Devaney, Robert L},
  journal={Transactions of the American Mathematical Society},
  volume={218},
  pages={89--113},
  year={1976}
}

@book{chow1994normal,
  title={Normal Forms and Bifurcation of Planar Vector Fields},
  author={Chow, Shui-Nee and Li, Chengzhi and Wang, Duo},
  year={1994},
  publisher={Cambridge University Press}
}

@book{bibikov2006local,
  title={Local Theory of Nonlinear Analytic Ordinary Differential Equations},
  author={Bibikov, Yuri N},
  volume={702},
  year={2006},
  publisher={Springer}
}

@book{siegel2012lectures,
  title={Lectures on Celestial Mechanics},
  author={Siegel, Carl L. and Moser, J{\"u}rgen K.},
  year={2012},
  publisher={Springer}
}

@article{mello2004saddle,
  title={A saddle-center bifurcation in planar Hamiltonian vector fields},
  author={Mello, Luis Fernando and Santos, Grasiele Batista dos},
  journal={Revista Brasileira de Ensino de Física},
  volume={26},
  pages={371--377},
  year={2004},
  publisher={SciELO}
}

@article{grotta2025global,
  title={Global normalizations for centers of planar vector fields},
  author={Grotta-Ragazzo, C. and Nascimento, Francisco Jos{\'e} dos Santos},
  journal={Journal of Differential Equations},
  volume={415},
  pages={701--721},
  year={2025},
  publisher={Elsevier}
}

@article{dosSantosNascimento2024,
  title={Conjugação analítica entre sistemas diferenciais planares e sistemas potenciais},
  author={Dos Santos Nascimento, Francisco Jos{\'e}},
  journal={CQD -- Revista Eletrônica Paulista de Matemática},
  pages={e24015},
  year={2024}
}

@article{jarque1994structural,
  title={Structural stability of planar Hamiltonian polynomial vector fields},
  author={Jarque, Xavier and Llibre, Jaume},
  journal={Proceedings of the London Mathematical Society},
  volume={69},
  number={3},
  pages={617--640},
  year={1994},
  publisher={Oxford University Press}
}

@phdthesis{nascimento2023sistemas,
  title={Sistemas Newtonianos revers{\~A}veis bidimensionais},
  author={Nascimento, Francisco Jos{\~A} and others},
  year={2023},
  school={Universidade de S{\~A}{\pounds} o Paulo}
}

@article{chavarriga1999survey,
  title={A survey of isochronous centers},
  author={Chavarriga, Javier and Sabatini, Marco},
  journal={Qualitative theory of Dynamical Systems},
  volume={1},
  pages={1--70},
  year={1999},
  publisher={Springer}
}
\end{document}